\documentclass{amsproc}

\newtheorem{theorem}{Theorem}[section]
\newtheorem{lemma}[theorem]{Lemma}
\newtheorem{proposition}[theorem]{Proposition}

\theoremstyle{definition}

\theoremstyle{remark}

\newtheorem{quest}[theorem]{Question}

\numberwithin{equation}{section}

\def\k{{\mathbb K}}

\def\ad{{\rm ad}}
\def\HH{{\rm HH}}

\def\uq{U_q^+(\mathfrak{so}_5)}

\begin{document}

\title{Is $A_1$ of type $B_2$?}

\author{S Launois}
\address{Universit\'e de Caen Normandie, 
UMR 6139 LMNO,
14032 Caen, France}
\email{stephane.launois@unicaen.fr}

\author{I Oppong}
\address{School of Computing and Mathematical Sciences, University of Greenwich,  Old Naval College, Park Row, Greenwich, London, SE10 9LS, United Kingdom}
\email{I.Oppong@greenwich.ac.uk}

\date{January 1, 1994 and, in revised form, June 22, 1994.}


\begin{abstract}
By a theorem of Dixmier, primitive quotients of enveloping algebras of finite-dimensional complex nilpotent Lie algebras are isomorphic to Weyl algebras. In view of this result, it is natural to consider simple quotients of positive parts of quantized enveloping algebras (and more generally of uniparameter Quantum Nilpotent Algebras) as quantum analogues of Weyl algebras. In this note, we study the Lie algebra of derivations of the simple quotients of $\uq$ of Gelfand-Kirillov dimension 2. For a specific family of such simple quotients, we prove that all derivations are inner (as in the case of Weyl algebras) whereas all other such algebras are quantum Generalized Weyl Algebras over a commutative Laurent polynomial algebra in one variable and have a first Hochschild cohomology group of dimension 1.
\end{abstract}

\maketitle

\section{Introduction}
Let $\k$ be a field of characteristic $0$ and let $q$ be a non-zero scalar which is not a root of unity. Weyl algebras are central objects in noncommutative algebra due to their appearences in a range of areas e.g. representation theory of Lie algebras; combinatorics; quantum mechanics; $\mathcal{D}$-modules to name a few. Despite many efforts over the last 60 years, the structure of Weyl algebras is still proving mysterious. For instance, only the automorphism group of the first Weyl algebra $A_1(\k)$ has been computed (by Dixmier in \cite{dixmier}) and only the irreducible representations of $A_1(\k)$ have been classified in a highly technical tour de force by Block \cite{block}. These two questions remain open for higher Weyl algebras. 

Since Weyl algebras are simple, their endomorphisms are injective, and it is an open conjecture of Dixmier \cite[Probl\`eme 1]{dixmier} that all endomorphisms are actually automorphisms. This conjecture is (stably) equivalent to the celebrated Jacobian Conjecture (see \cite{belov,tsuchimoto}), and it is still open even in the case of the first Weyl algebra $A_1(\k)$, to the best of the authors' knowledge. 

Given its central role in noncommutative algebra, it is not surprising that mathematicians have tried to generalise and/or deform Weyl algebras. This led in particular to the notion of Generalized Weyl Algebras (GWAs for short), a class of noncommutative algebras introduced by Bavula in the 1990s (see \cite{bavula}). This class of algebras includes in particular two families of (quantum) algebras that are deforming Weyl algebras: quantum Weyl algebras and Weyl-Hayashi algebras. The later appear as simple quotients of the quantum Heisenberg algebra $U_q^+ (\mathfrak{sl}_3)$ \cite{hayashi}, and were studied for instance in \cite{alevdumas,kirkmansmall}. A well-known theorem of Dixmier \cite[Th\'eor\`eme 4.7.9]{dixmierbook} 
shows that primitive quotients of enveloping algebras of nilpotent Lie algebras are isomorphic to Weyl algebras. In many respects, (uniparameter) Quantum Nilpotent Algebras (QNA for short), also known as CGL extensions \cite{gypnas,llr-ufd},  are behaving very much as enveloping algebras of nilpotent Lie algebras, and indeed include all quantum Schubert cells $U_q[w]$ which are deformations of the enveloping algebras of $\mathfrak{n}^+ \cap w(\mathfrak{n}^-)$, where $\mathfrak{n}^{\pm}$ are the nilradicals of the standard opposite Borel subalgebras of  a symmetrizable Kac-Moody algebra $ \mathfrak{g}$ and where $w$ is an element of the Weyl group of $ \mathfrak{g}$.  

As such, it is natural to think of primitive/simple quotients of QNAs as ``quantum'' Weyl algebras. This point of view is reinforced by the fact that primitive quotients of uniparameter QNAs have even Gelfand-Kirillov dimension \cite[Theorem 1.6]{bln2009}. With this idea in mind, Lopes \cite{lopes} studied the primitive ideals of $U_q^+(\mathfrak{sl}_n)$ coming from the $0$-stratum, that is the primitive ideals that do not contain any $H$-eigenvectors for the natural action by automorphisms of the torus $H:= (\k ^*)^{n-1}$ on $U_q^+ (\mathfrak{sl}_n)$. These primitive ideals are all maximal of even Gelfand-Kirillov dimension, and so feature some common properties with Weyl algebras. However, they also feature notable differences: they contain non-trivial units and not all derivations are inner in general. It is worth noting that, in the case of $\mathfrak{sl}_3$, these primitive quotients are the Weyl-Hayashi algebras and can be presented as GWAs over a (commutative) Laurent polynomial ring $\k[h^{\pm 1}]$. The presence of non-trivial units was exploited in \cite{ak1} to prove that all endomorphisms of Weyl-Hayashi algebras (and more generally of quantum GWAs over a commutative Laurent polynomial algebra) are automorphisms. In this note, we study the case of the positive part $U_q^+ (\mathfrak{so}_5)$ of the quantized enveloping algebra of $\mathfrak{so}_5$, and this will lead us to a ``quantum deformation of $A_1(\k)$''. Recently, the authors did study the case of $U_q^+ (G_2)$ \cite{lo} to obtain quantum/algebraic deformations of the second Weyl algebra. 

In \cite{sl}, primitive ideals of $U_q^+ (\mathfrak{so}_5)$ were classified. This led in particular to three non-isomorphic families of simple quotients of $U_q^+ (\mathfrak{so}_5)$ of Gelfand-Kirillov dimension 2. Out of these three families, we retrieve the Weyl-Hayashi algebras as a first family, another family of quantum GWAs over a commutative Laurent polynomial ring, and a third family $B_{\alpha, \beta}$, with $\alpha,\beta \in \k^*$, which cannot be presented as a GWA over a commutative Laurent polynomial ring. This last family can be presented as the $\k$-algebra generated by $e_1$, $e_2$, $e_3$ and $e_4$ subject to the following relations:
\begin{align*}
e_2e_1&=q^{-2}e_1e_2& e_4e_1&=q^2e_1e_4-q^2e_2&  e_4e_2&=q^{-2}e_2e_4+(1-q^{-2})\beta
\end{align*} 
and $$\frac{q^6}{q^4-1} e_2^2=  \dfrac{-(q+q^{-1})}{1-q^{-2}} \alpha -\beta e_1+e_1e_2e_4,$$
where $\alpha,\beta \in \k^*$.

Before we go any further, it is worth noting that these relations are just expressing the fact that  $B_{\alpha, \beta}$ is the quotient of $\uq$ by the maximal ideal $\langle \chi_1-\alpha, \chi_2-\beta\rangle$ of $\uq$, where $\chi_1$ and $\chi_2$ are the generators of the center of $\uq$ (see \cite[Proposition 2.7]{sl}).

For the first two families of simple quotients of $\uq$, since they can be presented as quantum GWAs over a Laurent polynomial ring, their endomorphisms were proved to be automorphisms in \cite{ak1}, and derivations were computed in \cite{ak}, proving in particular that not all derivations are inner (whereas all derivations of a Weyl algebra are inner). 

Our main aim in this note is to show that all derivations of $B_{\alpha, \beta}$ when $\alpha\beta\neq 0$ are inner. This makes the algebra $B_{\alpha, \beta}$ with $\alpha\beta\neq 0$ into a quantum deformation of the first Weyl algebra which is simple, has Gelfand-Kirillov dimension 2, has no non-trivial units, and has all its derivations inner. This somehow justifies the claim that ``$A_1(\k)$ is of type $B_2$'' since none of the primitive quotients of the quantum Heisenberg algebra $U_q^+(\mathfrak{sl}_3)$ have all these properties.

This paper is organised as follows. We recall some basic properties of $\uq$ and $B_{\alpha, \beta}$ in Sections \ref{ev1} and \ref{sec3} respectively. More importantly, we study Cauchon's deleting derivation algorithm \cite{ca} for the QNA $\uq$ in Section \ref{ev1}. This theory  allows us to construct an embedding: 
\begin{equation*}
B_{\alpha,\beta}\subset R=B_{\alpha,\beta}[e_4^{-1}] 
\end{equation*}
and to present $R$ as a GWA. This presentation of $R$ is obtained through the elements $f_1$ and $f_2$ constructed at the first step of the deleting derivation algorithm (DDA for short) for $\uq$. We then extend any derivation $D$ of $B_{\alpha,\beta}$ uniquely to a derivation of the GWA $R$. Thanks to a result of Kitchin \cite[Proposition 1.3]{ak}, derivations of $R$ are known to be the sum of an inner derivation and a scalar derivation (i.e. a derivation that acts on the generators $f_1,~f_2,~ e_3$ and $e_4$ of $R$ by multiplication by scalars). Having a perfect description of $D$ as a derivation $R$, we then describe it as a derivation of $B_{\alpha,\beta}$ by controlling the effect of the DDA  on our derivation $D$. Similarly to Weyl algebras, we conclude that every derivation of $B_{\alpha,\beta}$ is inner provided $\alpha\beta\neq 0.$

\section{Deleting derivation algorithm for the Quantum Nilpotent Algebra $\uq$}
\label{ev1}

\subsection{$\uq$ as a Quantum Nilpotent Algebra} 

The quantized enveloping algebra $\uq$ is the $\k$-algebra generated by two indeterminates $E_1$ and $E_4$ subject to the quantum Serre relations:
\begin{eqnarray*}
 & & E_1^2 E_4   - (q^2 + q^{-2} ) E_1 E_4 E_1 + E_4 E_1^2 = 0 \\
 & & E_4^3 E_1 -  (q^2 + 1 + q^{-2} ) E_4^2 E_1 E_4 + (q^2 + 1 + q^{-2} ) E_4 E_1 E_4^2  - E_1 E_4^3 = 0.
\end{eqnarray*}

To each reduced decomposition of the longest element of the Weyl group of $\mathfrak{so}_5$, we can associate a PBW basis for $\uq$ by using Lusztig automorphisms. In the present case, this leads us to introduce the following elements (see \cite[Appendix A]{DeGraaf2001} with $w_0=s_{\beta}s_{\alpha}s_{\beta}s_{\alpha}$, where the root vectors $E_{\beta}$, $E_{\alpha+\beta}$, $E_{2\alpha +\beta }$ and $E_{\alpha}$ from \cite{DeGraaf2001} corresponds to our $E_1$, $E_2$, $E_3$ and $E_4$ respectively):
$$E_2:=E_1E_4-q^{-2}E_4E_1$$
and $$E_3:=\frac{1}{q+q^{-1}}\left(E_2E_4-E_4E_2 \right).$$

By results of  Levendorskii-Soibelman and Lusztig (see \cite[Theorem I.6.8]{bg}), the monomials $E_1^{i_1}E_2^{i_2}E_3^{i_3}E_4^{i_4}$ with $i_1,i_2,i_3,i_4 \in \mathbb{Z}_{\geq 0}$ form a PBW basis of $\uq$, so that $\uq$ can be seen as the $\k$-algebra generated by  $E_1$, $E_2$, $E_3$ and $E_4$ subject to the relations below:
\begin{align*}
E_2E_1&=q^{-2}E_1E_2& E_4E_1&=q^2E_1E_4-q^2E_2& E_3E_1&=E_1E_3+\dfrac{q-q^3}{1+q^2} E_2^2\\ E_4E_2&=E_2E_4-(q+q^{-1})E_3&
E_3E_2&=q^{-2}E_2E_3& E_4E_3&=q^{-2}E_3E_4.
\end{align*} 

This presentation leads to the following presentation of $\uq$ as an iterated Ore extension: 
$$B:=\uq=\mathbb{K}[E_1][E_2;\sigma_2][E_3;\sigma_3,\delta_3][E_4;\sigma_4,\delta_4],$$ where $\sigma_2$ is an automorphism of $\mathbb{K}[E_1]$ defined by
$\sigma_2(E_1)=q^{-2}E_1;$
$\sigma_3$ is an automorphism of $\mathbb{K}[E_1][E_2;\sigma_2]$ defined by
$\sigma_3(E_1)=E_1, \ \sigma_3(E_2)=q^{-2}E_2;$
$\delta_3$ is a $\sigma_3$-derivation of $\mathbb{K}[E_1][E_2;\sigma_2]$  defined by
$\delta_3(E_1)=(q-q^3)/(1+q^2)E_2^2, \ \delta_3(E_2)=0;$
$\sigma_4$ is an automorphism of $\mathbb{K}[E_1][E_2;\sigma_2][E_3;\sigma_3,\delta_3]$ defined by
$\sigma_4(E_1)=q^2E_1, \ \sigma_4(E_2)=E_2, \ \sigma_4(E_3)=q^{-2}E_3$ and
$\delta_4$ is a $\sigma_4$-derivation of
 $\mathbb{K}[E_1][E_2;\sigma_2][E_3;\sigma_3,\delta_3]$ defined by
$\delta_4(E_1)=-q^2E_2, \ \delta_4(E_2)=-(q+q^{-1})E_3, \ \delta_4(E_3)=0.$
Note in particular that $B$ is a noetherian domain.

This presentation of $\uq$ as an iterated Ore extension makes $\uq$ into a so-called (uniparameter) QNA, (see for instance \cite[Proposition 6.1.2 and Lemme 6.2.1]{ca}). At this stage, we do not need to recall the definition of a QNA. However, we point out that, first, this implies that all primes of $B$ are completely prime, and then that this will allow us to apply the deleting derivation theory of Cauchon \cite{ca} to $\uq$ with this iterated Ore extension structure.  It is worth noting that the iterated Ore extension presentation of $\uq$ used in \cite{dumas,sl} does not allow to apply Cauchon's deleting derivation theory. 

\subsection{Deleting derivation algorithm of $\uq$}
\label{ddauq}
Since the iterated Ore extension $B$ is a QNA, similarly to \cite[Sec 2.3]{lo}, Cauchon's theory of deleting derivation algorithm (DDA for short) (see \cite{ca}) can be used to construct the following elements of the skew field  of fractions $\mathrm{Frac}(B)$ of $B$:  
\begin{align*}
E_{i,5}&:=E_i \qquad (i\in [1,4])&
E_{1,4}&=E_1+p_1E_2E_4^{-1}+p_2E_3E_4^{-2}\\
E_{2,4}&=E_2+p_3E_3E_4^{-1}&
E_{3,4}&=E_4 \\
 E_{4,4}&=E_4&
E_{1,3}&=E_{1,4}+p_4E_{2,4}^2E_{3,4}^{-1}\\
T_1:&=E_{1,2}=E_{1,3}&
T_2:&=E_{2,2}=E_{2,3}=E_{2,4}\\
T_3:&=E_{3,2}=E_{3,3}=E_{3,4}=E_3&
T_4:&=E_{4,2}=E_{4,3}=E_{4,4}=E_4,
\end{align*}
where 
\begin{align*} 
p_1=&\dfrac{1}{q^{-2}-1}=\dfrac{q^2}{1- q^{2}},&~~~ p_2=&\dfrac{q^{-1}+q^{-3}}{(1+q^{-2})(1-q^{-2})^2}=\dfrac{q^3}{(q^2-1)^2},\\
 p_3=&\dfrac{-(q+q^{-1})}{1-q^{-2}}=\dfrac{-(q^3+q)}{q^2-1}, &~~~p_4=&\dfrac{q-q^3}{(1-q^{-4})(1+q^2)}=\dfrac{-q^5}{(1+q^2)^2}. 
 \end{align*}

Recall from \cite[Lemma 3.1]{dumas} that the center of $B$ is a polynomial algebra $Z(B)=\k[\chi_1, \chi_2],$
where $$ \chi_1:=T_1T_3=E_{1,4}E_{3,4}+p_4E_{2,4}^2= E_1E_3+p_4E_2^2$$ and
$$\chi_2:=T_2T_4=E_2E_4+p_3E_3.$$

Note that these two central generators take a simpler form when expressed in the new elements $T_i$.

\section{Simple quotients of $U_q^+(B_2)$}
\label{sec3}

Primitive/maximal ideals of $\uq$ were classified in \cite{sl}, where it is proved (see \cite[Propositions 2.6 and 2.7]{sl}) that the only maximal ideals of height 2 are the ideals $\langle \chi_1-\alpha,\chi_2-\beta\rangle$ with $(\alpha,\beta )\in \mathbb{K}^2\setminus\{(0,0)\}$. 

Our aim in this note is to compute the derivations of the corresponding simple quotients of $\uq$. So we set:  
$$B_{\alpha,\beta}:=\frac{B}{\langle \chi_1-\alpha, \chi_2-\beta\rangle},$$ where $(\alpha,\beta )\in \mathbb{K}^2\setminus\{(0,0)\}.$ 

The algebra $B_{\alpha,\beta}$ is a simple noetherian domain with Gelfand-Kirillov dimension $2$ (\cite[Proposition 2.8]{sl}) and its center is reduced to $\k$ by \cite[Proposition 3.7]{sl}.  From \cite[Observation 3.1]{sl}, we also have that $B_{\alpha,\beta}$ is a quantum deformation of the first Weyl algebra $A_1(\k)$, hence $B_{\alpha,\beta}$ can be thought of as a quantum first Weyl algebra.  

In the case where $\alpha$ or $\beta$ is zero (but not both), then $B_{0,\beta}$ and $B_{\alpha,0}$ are Generalized Weyl Algebras (GWA) over a commutative Laurent polynomial ring in one variable by \cite[Propositions 3.9 and 3.10]{sl}. As such, their derivations were computed in \cite{ak}. Thus, from now on, we focus our attention on the case where $\alpha\beta \neq 0$.

Denote the canonical image of each $E_i$  by $e_i$  in  $B_{\alpha,\beta}.$  We have that $B_{\alpha,\beta}$ satisfies the following commutation relations:
\begin{align*}
e_2e_1&=q^{-2}e_1e_2& e_4e_1&=q^2e_1e_4-q^2e_2& e_3e_1&=e_1e_3+\dfrac{q-q^3}{1+q^2} e_2^2\\ e_3e_2&=q^{-2}e_2e_3& e_4e_2&=e_2e_4-(q+q^{-1})e_3& e_4e_3&=q^{-2}e_3e_4,
\end{align*}
and
\begin{align}
\begin{split}
  e_3&= k_1\beta-k_1e_2e_4 \\
e_2^2&= \alpha k_2-\beta k_1k_2e_1+k_1k_2e_1e_2e_4,
\end{split}\label{c5E0}
\end{align}
where $k_1=1/p_3=\dfrac{1-q^2}{q^3+q}$ and $k_2=1/p_4=-q^{-5}(1+q^2)^2.$

The following relations can easily be checked by induction.

\begin{lemma}
\label{c5l2} For each $i\in \mathbb{N}$: 
\begin{enumerate}
\item $e_3e_1^i=e_1^ie_3+
d[i]e_1^{i-1}e_2^2,$
 where $d[i]=\dfrac{(q-q^3)(1-q^{-4i})}{(1-q^{-4})(1+q^2)}.$ Note in particular that $d[0]=0.$ 
\item $e_4e_1^i=q^{2i}e_1^ie_4-f[i]e_1^{i-1}e_2,$ 
 where $f[i]=\dfrac{q^{2i}(1-q^{-4i})}{1-q^{-4}}.$ Note in particular that $f[0]=0.$
\end{enumerate}
\end{lemma}

A linear basis of $B_{\alpha,\beta}$ was obtained in \cite{sl}. Noting that our generators $E_1$, $E_2$ and $E_4$ of $\uq$ correspond to the generators $e_1$, $e_3$ and $e_2$ of $U_{q^{-1}}^+(B_2)$ in \cite{sl}, the linear basis of $B_{\alpha,\beta}$ from  \cite[Proposition 3.2]{sl} can be expressed as follows with the notation of this note. 
\begin{proposition}
\label{c5E1}
The set $\mathcal{E}=\{e_1^ie_4^j, e_1^ie_2e_4^j\mid i,j\in \mathbb{N}\}$ is a $\mathbb{K}$-basis of $B_{\alpha,\beta}.$
\end{proposition}

\section{Derivations of $B_{\alpha,\beta}$} 
\label{der}
In this section, we aim to compute the set of $\k$-derivations of the algebra $B_{\alpha,\beta}.$  Our strategy is to prove that a localization of $B_{\alpha,\beta}$ is a quantum GWA. This isomorphism was obtained through the first step of the DDA since this GWA presentation is naturally expressed in terms of the elements produced by the first step of the DDA, that is $E_{1,4}$, $E_{2,4}$, $E_3$ and $E_4$. 

Once this isomorphism is established, we can extend any derivation of $B_{\alpha,\beta}$ to a derivation of this GWA. It then follows from work of Kitchin \cite{ak} that derivations of a (simple) quantum GWA are the sum of an inner derivation and a central/scalar derivation. To conclude, we will need to control the effect of the deleting derivation algorithm on our derivation. 

Unless otherwise stated, we assume that $\alpha \beta \neq 0$.

\subsection{Embedding $B_{\alpha,\beta}$ into a quantum GWA}

At this stage, we recall that the (quantum) generalized Weyl algebra $\k[h^{\pm 1}](\sigma_q,a)$ associated to the automorphism $\sigma_q$ of $\k[h^{\pm 1}]$ defined by $\sigma_q(h)=q^2h$ and to a nonzero element $a \in \k[h^{\pm 1}]$ is the $\k[h^{\pm 1}]$-algebra generated by two indeterminates $x$ and $y$ subject to the relations
$$yx=a,~ xy=\sigma_{q}(a),~ xh=q^2hx, \mbox{ and }yh=q^{-2}hy.$$ 

Observe that $E_4 \notin \langle \chi_1-\alpha,\chi_2-\beta\rangle$. Indeed, otherwise, $E_4$ and $[E_4,E_2]=-(q+q^{-1})E_3$ would belong to $ \langle \chi_1-\alpha,\chi_2-\beta\rangle$ and then $\chi_2= E_2E_4+p_3E_3$ would also belong to   $ \langle \chi_1-\alpha,\chi_2-\beta\rangle$, a contradiction since $\beta \neq 0$. 

Set $e_4:= E_4+  \langle \chi_1-\alpha,\chi_2-\beta\rangle \in B_{\alpha,\beta}$. Since the multiplicative system generated by $E_4$ in $B$ satisfies the Ore condition on both side by \cite[Lemme 2.1]{ca}, it follows that  $\{e_4^i \mid i\in \mathbb{N}\}$ is a multiplicative system of regular elements of $B_{\alpha,\beta}$ satisfying the Ore condition on both sides. 

We set $R:= B_{\alpha,\beta}[e_4^{-1}]=\frac{B[E_4^{-1}]}{ \langle \chi_1-\alpha,\chi_2-\beta\rangle [E_4^{-1}]}$. Our aim in this section is to prove that $R$ is a GWA of the form $\k[h^{\pm 1}](\sigma_q,a)$. To prove this, we will make use of the elements created by the first step of the DDA. Namely, recall that the first step of the algorithm leads to the following elements of $\mathrm{Frac} (B)$: 
\begin{align*}
E_{1,4}&=E_1+p_1E_2E_4^{-1}+p_2E_3E_4^{-2} \mbox{ and }E_{2,4}=E_2+p_3E_3E_4^{-1}.&
\end{align*}
Recall also that $$ \chi_1=E_{1,4}E_{3}+p_4E_{2,4}^2 \mbox{ and }
\chi_2=E_{2,4}E_4.$$ 

This leads us to introduce the following elements of $R$: 
$$f_1:= e_1+p_1e_2e_4^{-1}+p_2e_3e_4^{-2} $$
and 
$$f_2:= e_2+p_3e_3e_4^{-1}.$$

Clearly, $f_1,~f_2, e_3, e_4^{\pm 1}$ can be expressed in terms of the generators $e_1,~e_2,~e_3,~e_4^{\pm 1}$ of $R$. 
Note that the converse is also true since $e_2= f_2-p_3e_3e_4^{-1}$ and $e_1= f_1-p_1e_2e_4^{-1}-p_2e_3e_4^{-2} = f_1-p_1(f_2-p_3e_3e_4^{-1})e_4^{-1}-p_2e_3e_4^{-2}$. 

At this point, it is worth noting that the fact that $\chi_2 = \beta$ in $ B_{\alpha,\beta}$ forces $f_2$ to be invertible in $R$ and $e_4:= \beta f_2^{-1}$. As a consequence, $R$ is generated by $f_1,~f_2^{\pm 1}$ and $ e_3$, and one can easily show using various universal properties that 
$$R \simeq \frac{ \k\langle f_1,f_2^{\pm 1}, e_3\rangle }{\langle f_1 f_2 = q^2 f_2 f_1, ~ e_3f_2=q^{-2}f_2e_3, ~ e_3f_1=f_1e_3+\frac{q-q^3}{1+q^2}f_2^2, ~f_1e_3-\frac{q^5}{(1+q^2)^2}f_2^2=\alpha \rangle }.$$

We are ready to prove that $R$ is indeed a quantum GWA.

\begin{proposition}
\label{RGWA}
$R$ is isomorphic to the GWA $\k[h^{\pm 1}](\sigma_q,\alpha+\frac{q}{(q^2+1)^2}h^2)$ by an isomorphism that sends $f_2$ to $h$, $f_1$ to $x$ and $e_3$ to $y$.
\end{proposition}
\begin{proof}
By universal property, it is easy to check that there exists an homomorphism $$\phi: \k[h^{\pm 1}](\sigma_q,\alpha+\frac{q}{(q^2+1)^2}h^2) \rightarrow R$$ defined by $\phi(h)=f_2$, $\phi(x)=f_1$ and $\phi (y)=e_3$. By the above discussion, this homomorphism is surjective. Moreover, the Gelfand-Kirillov dimension of the domains $R$ and $\k[h^{\pm 1}](\sigma_q,\alpha+\frac{q}{(q^2+1)^2}h^2)$ is equal to 2. This follows from \cite[Proposition 2.8]{sl} for $R$. For the GWA $\k[h^{\pm 1}](\sigma_q,\alpha+\frac{q}{(q^2+1)^2}h^2)$, this follows from the fact that it contains the quantum plane $\k_{q^2}[h,x]$ and is contained in the quantum torus $\k_{q^2}[h^{\pm 1},x^{\pm 1}]$ (and both have Gelfand-Kirillov dimension 2). Hence $\phi$ is an isomorphism by \cite[Proposition 3.15]{KrauseLenagan}.
\end{proof}

As a GWA, the algebra $R$ has a linear basis formed by the monomials $\{f_2^if_1^j\mid i,j \in \mathbb{N}\} \cup \{f_2^ie_3^j \mid i,j \in \mathbb{N} \mbox{ with }j\geq 1\}$. However, it also possesses a linear basis inherited from $B_{\alpha,\beta}$, see Proposition \ref{c5E2}. 

\begin{proposition}
\label{c5E2}
The set $\mathcal{E}'=\{e_1^ie_4^j, e_1^ie_2e_4^j\mid (i,j)\in \mathbb{N} \times \mathbb{Z}\}$ is a $\mathbb{K}$-basis of $R$.
\end{proposition}

\subsection{Derivations of $B_{\alpha,\beta}$ when $\alpha\beta\neq 0$}

Let Der$(B_{\alpha,\beta})$ denote the set of $\mathbb{K}$-derivations of $B_{\alpha,\beta}$ and $D\in \text{Der}(B_{\alpha,\beta}).$ One can extend $D$ to a derivation of the GWA $R$ (still denoted by $D$). Derivations of quantum GWA have been described in \cite{ak} and we deduce from the isomorphism $\phi$ from Proposition \ref{RGWA} the following result:

 \begin{lemma} 
\label{L1}
There exist $x \in R$ and $\lambda \in \k$ such that $D=\mathrm{ad}_x + \delta_{\lambda}$, where $\delta_{\lambda} \in \mathrm{Der}(R) $ is defined by $\delta_{\lambda}(f_2)=0$, $\delta_{\lambda}(f_1)=\lambda f_1$ and $\delta_{\lambda} (e_3)=-\lambda e_3$.

Moreover $\delta_{\lambda}(e_4)=0$.
\end{lemma}
 \begin{proof}
 This follows from \cite[Proposition 2.3]{ak} and the fact that $e_4= \beta f_2^{-2}$.
 \end{proof}

We are now ready to describe $D$ as a derivation of $B_{\alpha,\beta}.$
\begin{lemma}
\label{5cl3}
\begin{enumerate}
\item $x\in B_{\alpha,\beta}$;
\item $\lambda=0$;
\item  $D=\ad_x.$
\end{enumerate}
\end{lemma}
\begin{proof}
(1) Since $x\in R=B_{\alpha,\beta}[e_4^{-1}],$ it follows from Proposition \ref{c5E2} that one can write  $$x=\displaystyle\sum_{(i,j)\in I_1}a_{i,j}e_1^i e_4^j +\displaystyle\sum_{(i,j)\in I_2}b_{i,j}e_1^i e_2e_4^j,$$ where $a_{i,j},b_{i,j} \in \k$ are families of scalars and, $I_1$ and $I_2$ are finite subsets of $\mathbb{N}\times \mathbb{Z}.$
Write $x=x_++x_-,$ where $$x_+=\displaystyle\sum_{\substack{(i,j)\in I_1\\ j\geq 0}}a_{i,j}e_1^i e_4^j
+\displaystyle\sum_{\substack{(i,j)\in I_2\\j\geq 0}}b_{i,j}e_1^i e_2e_4^j$$ and $$x_-=\displaystyle\sum_{\substack{(i,j)\in I_1\\ j<0}}a_{i,j}e_1^i e_4^j
+\displaystyle\sum_{\substack{(i,j)\in I_2\\j<0}}b_{i,j}e_1^i e_2e_4^j.$$ Clearly, $x_+\in B_{\alpha,\beta}.$ We want to show that $x_-=0.$ Suppose that
 $x_-\neq 0.$ Then, there exists a minimum $j_0<0$ such that $a_{i,j_0}\neq 0$ or $b_{i,j_0}\neq 0$ for some $i\geq 0,$ and $a_{i,j}= b_{i,j}=0$ for all $j<j_0.$ 
 Write $$x_-=\displaystyle\sum_{\substack{(i,j)\in I_1\\ j_0\leq j<0}}a_{i,j}e_1^i e_4^j
+\displaystyle\sum_{\substack{(i,j)\in I_2\\j_0\leq j<0}}b_{i,j}e_1^i e_2e_4^j.$$ 

Now, $D(e_3)=\text{ad}_{x_+}(e_3)+\text{ad}_{x_-}(e_3)+\delta_{\lambda}(e_3)\in B_{\alpha,\beta}.$ Since $\delta_{\lambda}(e_3)=-\lambda e_3\in B_{\alpha,\beta}$, we obtain that $w:=\text{ad}_{x_-}(e_3)=x_-e_3-e_3x_-\in B_{\alpha,\beta}.$ This implies that $we_4^{-j_0-1}\in B_{\alpha,\beta}$, that is:
\begin{align*}
&\displaystyle\sum_{\substack{(i,j)\in I_1\\ j_0\leq j<0}}a_{i,j}e_1^i e_4^j e_3e_4^{-j_0-1} 
+\displaystyle\sum_{\substack{(i,j)\in I_2\\j_0\leq j<0}}b_{i,j}e_1^i e_2e_4^j e_3e_4^{-j_0-1} \\
- & \displaystyle\sum_{\substack{(i,j)\in I_1\\ j_0\leq j<0}}a_{i,j}e_3e_1^i e_4^{j-j_0-1} 
-\displaystyle\sum_{\substack{(i,j)\in I_2\\j_0\leq j<0}}b_{i,j}e_3e_1^i e_2e_4^{j-j_0-1} \in B_{\alpha,\beta}.
\end{align*}
Using the relation $e_3 e_4 =q^2 e_4 e_3$, this leads to:
\begin{align*}
&\displaystyle\sum_{\substack{(i,j)\in I_1\\ j_0\leq j<0}}a_{i,j}q^{-2j}e_1^i e_3 e_4^{j-j_0-1} 
+\displaystyle\sum_{\substack{(i,j)\in I_2\\j_0\leq j<0}}b_{i,j}q^{-2j} e_1^i e_2e_3 e_4^{j-j_0-1} \\
- &\displaystyle\sum_{\substack{(i,j)\in I_1\\ j_0\leq j<0}}a_{i,j}e_3e_1^i e_4^{j-j_0-1} 
-\displaystyle\sum_{\substack{(i,j)\in I_2\\j_0\leq j<0}}b_{i,j}e_3e_1^i e_2e_4^{j-j_0-1}  \in B_{\alpha,\beta}.
\end{align*}
The only terms that could not belong to $ B_{\alpha,\beta}$ in the above sums are those for which the exponents of $e_4$ is negative. Since $j -j_0-1 \geq 0$ for all $j>j_0$, we obtain that:
\begin{align*} 
 & \displaystyle\sum_{(i,j_0)\in I_1}a_{i,j_0}q^{-2j_0}e_1^i e_3 e_4^{-1} 
+\displaystyle\sum_{(i,j_0)\in I_2}b_{i,j_0}q^{-2j_0} e_1^i e_2e_3 e_4^{-1} \\
- & \displaystyle\sum_{(i,j_0)\in I_1}a_{i,j_0}e_3e_1^i e_4^{-1} 
-\displaystyle\sum_{(i,j_0)\in I_2}b_{i,j_0}e_3e_1^i e_2e_4^{-1} \in B_{\alpha,\beta}.
\end{align*}
This is equivalent to 
\begin{align*} 
& \displaystyle\sum_{(i,j_0)\in I_1}a_{i,j_0}q^{-2j_0}e_1^i e_3 
+\displaystyle\sum_{(i,j_0)\in I_2}b_{i,j_0}q^{-2j_0} e_1^i e_2e_3  \\
- & \displaystyle\sum_{(i,j_0)\in I_1}a_{i,j_0}e_3e_1^i 
-\displaystyle\sum_{(i,j_0)\in I_2}b_{i,j_0}e_3e_1^i e_2\in B_{\alpha,\beta}e_4.
\end{align*}
Since $e_3 -k_1\beta \in  B_{\alpha,\beta}e_4$ by \eqref{c5E0}, this leads to 
\begin{align*} 
& \displaystyle\sum_{(i,j_0)\in I_1}a_{i,j_0}q^{-2j_0}k_1 \beta e_1^i  
+\displaystyle\sum_{(i,j_0)\in I_2}b_{i,j_0}q^{-2j_0} k_1 \beta e_1^i e_2 \\
- & \displaystyle\sum_{(i,j_0)\in I_1}a_{i,j_0}e_3e_1^i 
-\displaystyle\sum_{(i,j_0)\in I_2}b_{i,j_0}e_3e_1^i e_2\in B_{\alpha,\beta}e_4.
\end{align*}
Using the straightening relation between $e_3$ and $e_1^i$ from Lemma  \ref{c5l2} and the relation $e_3 e_2 =q^{-2}e_2e_3$, we get: 
\begin{align*}
&\displaystyle\sum_{(i,j_0)\in I_1}a_{i,j_0}q^{-2j_0}k_1 \beta e_1^i  
+\displaystyle\sum_{(i,j_0)\in I_2}b_{i,j_0}q^{-2j_0} k_1 \beta e_1^i e_2
- \displaystyle\sum_{(i,j_0)\in I_1}a_{i,j_0}e_1^ie_3 \\
- &\displaystyle\sum_{(i,j_0)\in I_1}a_{i,j_0}d[i]e_1^{i-1}e_2^2 
-\displaystyle\sum_{(i,j_0)\in I_2}b_{i,j_0}q^{-2}e_1^i  e_2e_3  -\displaystyle\sum_{(i,j_0)\in I_2}b_{i,j_0}d[i]e_1^{i-1}  e_2^3\in B_{\alpha,\beta}e_4.
\end{align*}
Since $e_3 -k_1\beta \in  B_{\alpha,\beta}e_4$ and $e_2^2 -\alpha k_2 +\beta k_1k_2 e_1 \in  B_{\alpha,\beta}e_4$ by \eqref{c5E0}, we obtain: 
\begin{align*}
&w':=\displaystyle\sum_{(i,j_0)\in I_1}a_{i,j_0}q^{-2j_0}k_1 \beta e_1^i  
+\displaystyle\sum_{(i,j_0)\in I_2}b_{i,j_0}q^{-2j_0} k_1 \beta e_1^i e_2
- \displaystyle\sum_{(i,j_0)\in I_1}a_{i,j_0}k_1\beta e_1^i \\
 - & \displaystyle\sum_{(i,j_0)\in I_1}a_{i,j_0}d[i]\alpha k_2 e_1^{i-1}
+ \displaystyle\sum_{(i,j_0)\in I_1}a_{i,j_0}d[i] \beta k_1 k_2 e_1^{i}  -\displaystyle\sum_{(i,j_0)\in I_2}b_{i,j_0}q^{-2}k_1\beta e_1^i  e_2 \\
 - & \displaystyle\sum_{(i,j_0)\in I_2}b_{i,j_0}d[i]\alpha k_2 e_1^{i-1} e_2 
+ \displaystyle\sum_{(i,j_0)\in I_2}b_{i,j_0}d[i] q^{-2}  \beta k_1 k_2 e_1^i  e_2
\in B_{\alpha,\beta}e_4.
\end{align*}
(Notice that a coefficient $q^{-2}$ has appeared in the last sum since we had to use the relation $e_2e_1=q^{-2}e_1e_2$.)

Note that all the monomials involved in the above sums are elements of the basis $\mathcal{E}$ of $B_{\alpha,\beta}$ not involving $e_4$. However, for an element to belong to $B_{\alpha,\beta}e_4$, it needs to have all the coefficients of the basis elements $e_1^i$ and $e_1^i e_2$ to be equal to $0$. So, if $i$ is maximum such that $(i,j_0)\in I_1$ and $a_{i,j_0}\neq 0$, we obtain by inspection of the coefficient of $e_i^i$ in $w'$:
$$a_{i,j_0}q^{-2j_0}k_1 \beta - a_{i,j_0}k_1\beta + a_{i,j_0}d[i] \beta k_1 k_2 = 0.$$
Since $k_1 \beta \neq 0$ and $d[i]  k_2 = 1-q^{-4i}$, this leads to 
$$a_{i,j_0}(q^{-2j_0} -q^{-4i})=0.$$
Since $q$ is not a root of unity and $j_0<0$ whereas $i\geq 0$, we conclude that $a_{i,j_0}=0$. This is a contradiction, and so all coefficients $a_{i,j_0}$ are equal to $0$.

By definition of $j_0$, this means that there exists $i\in \mathbb{N}$ such that $b_{i,j_0} \neq 0$. Let choose such $i$ maximal. Then the coefficients of $e_1^ie_2$ in $w'$ must be equal to $0$ for $w'$ to belong to $B_{\alpha,\beta}e_4 $. This leads to 
$$b_{i,j_0}q^{-2j_0} k_1 \beta - b_{i,j_0}q^{-2}k_1\beta +b_{i,j_0}d[i] q^{-2} \beta k_1 k_2=0. $$
Since $k_1 \beta \neq 0$ and $d[i]  k_2 = 1-q^{-4i}$, this leads to 
$$b_{i,j_0} \left( q^{-2j_0}  - q^{-2}+ q^{-2}(1-q^{-4i})\right)=0, $$
that is
$$b_{i,j_0}( q^{-2j_0}  -q^{-4i-2})=0. $$
Again, $q^{-2j_0}  -q^{-4i-2} \neq 0$ since $q$ is not a root of unity, and $-2j_0 > 0$ while $-4i-2 <0$. Hence, $b_{i,j_0}=0$, a contradiction with our choice of $j_0$. 

Hence $x_-=0$, as desired.

(2) Recall  that $\delta_{\lambda} (e_3)=-\lambda e_3.$ Also, from Lemma \ref{L1}, $\delta_{\lambda}(e_4)=\delta(e_2)=0$ and 
$\delta_{\lambda}(f_1)=\lambda f_{1}$.  Recall that  $e_2=f_2-p_3e_3e_4^{-1}$ in $R.$ Therefore, $\delta_{\lambda}(e_2)=-p_3\lambda e_3e_4^{-1}$. Now,
$D(e_2)=\text{ad}_x(e_2)+\delta_{\lambda}(e_2)\in B_{\alpha,\beta}.$ Since $\text{ad}_x(e_2)\in B_{\alpha,\beta},$ we have that
 $\delta_{\lambda}(e_2)=-p_3\lambda e_3e_4^{-1}\in B_{\alpha,\beta}. $ This implies that $p_3\lambda e_3\in B_{\alpha,\beta}e_4.$ Hence, $p_3\lambda k_1\beta-p_3\lambda k_1e_2e_4\in B_{\alpha,\beta}e_4$ since $e_3=k_1\beta-k_1e_2e_4$ by \eqref{c5E0}. Clearly, $e_2e_4\in B_{\alpha,\beta}e_4,$ it follows that  $p_3\lambda k_1\beta\in B_{\alpha,\beta}e_4.$  Since $k_1,p_3, \beta \neq 0,$ this forces $\lambda=0.$ Otherwise, we would have a contradiction given the basis $\mathcal{E}$ of $B_{\alpha,\beta}$ from Proposition \ref{c5E1}.

(3) Since $\lambda=0$, then $\delta_{\lambda}$ coincides with the $0$ derivation on a generating set for $R$. Hence $\delta_{\lambda}=0$ and $D=\mathrm{ad}_x$ with $x\in B_{\alpha,\beta}$ as desired.
\end{proof}

\subsection{Summary}

We are now ready to state the main result of this note which shows that the simple quotients $B_{\alpha,\beta}$ present similarities with the first Weyl algebra when $\alpha \beta \neq 0$. Before we proceed, let us recall that the first Hochschild cohomology group of $B_{\alpha,\beta}$ (denoted by $\HH^1(B_{\alpha,\beta})$) is defined by $$\HH^1(B_{\alpha,\beta}):=\frac{\text{Der}(B_{\alpha,\beta})}{\text{InnDer}(B_{\alpha,\beta})},$$
where the subset $\text{InnDer}(B_{\alpha,\beta}):=\{\text{ad}_x\mid x\in B_{\alpha,\beta}\}$ of $\text{Der}(B_{\alpha,\beta})$
is the set of inner derivations of $B_{\alpha,\beta}.$

 Recall that $\HH^1(B_{\alpha,\beta})$ is a free module over $Z(B_{\alpha,\beta})=\k.$

\begin{theorem}
Assume that $\alpha, \beta \neq 0$. 
\begin{enumerate}
\item  Every derivation of $B_{\alpha,\beta}$ is inner and $\HH^1(B_{\alpha,\beta})=0.$
\item $\dim\HH^1(B_{\alpha, 0})= \dim \HH^1(B_{0,\beta})=1.$
\end{enumerate}
 \end{theorem}
\begin{proof}
The first point was proved in Lemma \ref{5cl3}. For the second point, observe that $B_{\alpha,0}$ and $B_{0,\beta}$ are both quantum GWAs by \cite[Propositions 3.9 and 3.10]{sl2} and so the result follows from \cite[Proposition 1.3]{ak}.
\end{proof}

\section{Open questions}

As explained in the introduction, simple quotients of (uniparameter) QNAs can be thought of as ``algebraic deformation'' of Weyl algebras. However, some of them may behave differently to Weyl algebras. For instance, they may contain nontrivial units, and this fact was for instance used to prove that endomorphisms of GWAs over a Laurent polynomial ring in one variable  are automorphisms (see \cite{ak1}). Closer to the main topic of this paper, some simple quotients have only inner derivations whereas some have larger first Hochschild cohomology group. This leads us to the following natural questions: 

\begin{quest} 
\label{question 1}
Can we classify simple quotients of uniparameter QNAs for which all derivations are inner? 
\end{quest}


\begin{quest}
If $A$ is a simple quotient of a uniparameter QNA, are all endomorphisms of $A$ automorphisms? 
\end{quest}
 
General results regarding Question \ref{question 1} have been obtained while preparing this note and will be published in a forthcoming publication \cite{llo}.

\bibliographystyle{amsalpha}

\begin{thebibliography}{10}

%
\bibitem{alevdumas} 
J.~Alev and F.~Dumas.
\newblock \textit{Rigidit\'e des plongements des quotients primitifs minimaux de $U_q(sl(2))$ dans l'alg\`ebre quantique de Weyl-Hayashi.}
\newblock Nagoya Math. J. \textbf{143} (1996), 119--146.
%

%
\bibitem{dumas}
N. Andruskiewitsch and F. Dumas.
\newblock \textit{On the automorphism of $U_q^+(\mathfrak{g})$.} 
\newblock  In: Quantum Groups. IRMA Lect.
Math. Theor. Phys., European Mathematical Society, Z\"urich 12: 107--133, 2008.
%

%
\bibitem{bavula}
V.V.~Bavula.
\newblock \textit{Generalized Weyl algebras and their representations.}
\newblock Algebra i Analiz \textbf{4}(1) (1992), 75--97; translation in St. Petersburg Math. J. \textbf{4}(1) (1993), 71--92.
%

%
\bibitem{bln2009}
J.~Bell, S.~Launois and N.~Nguyen.
\newblock \textit{Dimension and enumeration of primitive ideals in quantum algebras.}
\newblock J. Algebr. Comb. \textbf{29} (2009), 269--294.
%

%
\bibitem{belov}
 A.~Belov-Kanel and M.~Kontsevich
 \newblock \textit{The Jacobian conjecture is stably equivalent to the Dixmier conjecture. }
 \newblock Mosc. Math. J. \textbf{7}(2) (2007), 209--218.
%


%
\bibitem{block}
R.~E. Block.
\newblock \textit{The irreducible representations of the Lie algebra $\mathfrak{sl}(2)$ and of the Weyl algebra. }
\newblock  Adv. Math. \textbf{39}(1) (1981), 69--110.
%


%
\bibitem{bg}
K.~A. Brown and K.~R. Goodearl.
\newblock \textit{Lectures on Algebraic Quantum Groups}.
\newblock Advanced Courses in Mathematics CRM Barcelona (Birkh\"auser, Basel, 2002).
%

%
\bibitem{ca}
G.~Cauchon.
\newblock \textit{Effacement des d{\'e}rivations et spectres premiers des alg\`ebres quantiques.}
\newblock {J. Algebra} \textbf{260} (2003), 476--518.
%

%
\bibitem{DeGraaf2001}
W.~A.~De Graaf. 
\newblock \textit{Computing with quantized enveloping algebras: PBW-type bases, highest-weight modules and R-matrices.}
\newblock J. Symbolic Comput. \textbf{32} (2001), 475--490.
%

%
\bibitem{dixmier}
J.~Dixmier.
\newblock \textit{Sur les alg\`ebres des Weyl.}
\newblock Bull. Soc. Math. France \textbf{96} (1968), 209--242.
%

%
\bibitem{dixmierbook}
J.~Dixmier.
\newblock \textit{Alg\`ebres enveloppantes.} 
\newblock Gauthier-Villars, Paris, 1974.
%

%
\bibitem{gypnas}
K.~R. Goodearl and M. ~Yakimov.
\newblock \textit{Quantum cluster algebras and quantum nilpotent algebras.}
\newblock Proc. Natl. Acad. Sci. USA \textbf{111}(27) (2014), 9696--9703.
%

%
\bibitem{hayashi}
T.~Hayashi.
\newblock \textit{$q$-analogues of Clifford and Weyl algebras---spinor and oscillator representations of quantum enveloping algebras.}
\newblock Comm. Math. Phys. \textbf{127}(1) (1990), 129--144.
%

%
\bibitem{kirkmansmall}
E.~E. Kirkman and L.~W. Small.
\newblock \textit{$q$-analogs of harmonic oscillators and related rings.}
\newblock Israel J. Math. \textbf{81} (1993), 111--127.
%

%
\bibitem{ak}
A.~P. Kitchin.
\newblock \textit{Derivations of Quantum and Involution Generalized Weyl Algebras.}
\newblock J. Algebra Appl. \textbf{23}(4) (2024),  article 2450076.
%


%
\bibitem{ak1}
A.~P. Kitchin and S.~Launois.
\newblock \textit{Endomorphisms of quantum generalized Weyl algebras.}
\newblock Lett. Math. Phys. \textbf{104}(7) (2014), 837--848,.
%

%
\bibitem{KrauseLenagan}
G.~R. Krause  and  T.~H. Lenagan.
\newblock \textit{Growth  of  algebras  and  Gelfand-Kirillov  dimension (revised ed.).}
\newblock Graduate Studies in Mathematics, vol. 22, American Mathematical Society, Providence, RI, 2000.
%


%
\bibitem{sl2}
S.~Launois.
\newblock \textit{On the automorphism groups of q-enveloping algebras of nilpotent Lie algebras.}
\newblock From Lie Algebras to Quantum Groups, 125--143, 2006.
%

%
\bibitem{sl}
S.~Launois.
\newblock \textit{Primitive ideals and automorphism group of ${U_q^+(B_2)}$.}
\newblock J. Algebra Appl. \textbf{6} (2007), 21--47.
%


%
\bibitem{llr-ufd}
S. Launois, T.~H. Lenagan and L. Rigal.
\newblock \textit{Quantum unique factorisation domains.}
\newblock J. London Math. Soc. (2) \textbf{74} (2006), 321--340.
%

%
\bibitem{llo}
S. Launois, S.~A. Lopes and I. Oppong.
\newblock \textit{On the first Hochschild cohomology group of quantum nilpotent algebras.}
\newblock In preparation.
%

%
\bibitem{lo}
S. Launois and I. Oppong.
\newblock \textit{Derivations of a family of quantum second Weyl algebras.}
\newblock Bull. Sci. Math. \textbf{184} (2023), article 103257.
%

%
\bibitem{lopes}
S.~A. Lopes.
\newblock \textit{Primitive ideals of $U_q(\mathfrak{sl}_n^+)$.}
\newblock Comm. Algebra \textbf{34}(12) (2006), 4523--4550.
%

%
\bibitem{tsuchimoto}
 Y.~Tsuchimoto. 
 \newblock \textit{Preliminaries on Dixmier conjecture. }
 \newblock Mem. Fac. Sci. Kochi Univ. Ser. A Math. \textbf{24} (2003), 43--59,.
 %
 
\end{thebibliography}

\end{document}